\documentclass{amsart}[12pt]
\usepackage{graphicx,graphics, amsmath, amsthm, amscd, amsfonts}

\usepackage{latexsym}
\makeatletter \oddsidemargin.9375in \evensidemargin \oddsidemargin
\newtheorem{theorem}{Theorem}[section]

\newtheorem{proposition}[theorem]{Proposition}

\theoremstyle{definition}

\numberwithin{equation}{section}

\begin{document}
\Large
\title[Maps preserving the spectrum of certain products of
operators]{Maps preserving the spectrum of certain products of operators}

\author[Ali Taghavi and Roja Hosseinzadeh]{Ali Taghavi and Roja
Hosseinzadeh}

\address{{ Department of Mathematics, Faculty of Basic Sciences,
 University of Mazandaran, P. O. Box 47416-1468, Babolsar, Iran.}}

\email{Taghavi@nit.ac.ir,  ro.hosseinzadeh@umz.ac.ir}

\subjclass[2000]{46J10, 47B48}

\keywords{Operator algebra; Non-linear spectrum preserving; Self-adjoint
operator}

\begin{abstract}\large
Let $\mathcal{H}$ be a complex Hilbert space, $\mathcal{B(H)}$ and
$\mathcal{S(H)}$ be the spaces of all bounded operators and
all self-adjoint operators on $\mathcal{H}$, respectively. We give the
concrete forms of the maps on $\mathcal{B(H)}$ and also
$\mathcal{S(H)}$ which preserve the spectrum of certain
products of operators.
\end{abstract} \maketitle

\section{Introduction And Statement of the Results}
\noindent
The study of spectrum-preserving linear maps between Banach
algebras goes back to Frobenius $[3]$ who studied linear maps on
matrix algebras preserving the determinant. The following
conjecture seems to be still open: Any spectrum-preserving linear
map from a unital Banach algebra onto a unital semi-simple Banach
algebra that preserves the unit is a Jordan morphism. There are
many other papers concerning this type of linear preservers; for
example, see $[1,6,8,10,12,13]$.
\par Without assuming linearity, spectrum-preserving maps are almost
arbitrary; see $[2,5,7,9,11]$. In $[9]$, Molnar considered
multiplicatively spectrum-preserving surjective maps on Banach
algebras in the sense that the spectrum of the product of the
image of any two elements is equal to the spectrum of the product
of those two elements, and proved that the maps are almost
isomorphisms in the sense that isomorphisms multiplied by a
signum function for the Banach algebra of all complex-valued
continuous functions on a first countable compact Hausdorff
space. In $[5]$, the authors considered the same problem for
triple Jordan products of operators and proved that such maps
must be Jordan isomorphisms multiplied by a cubic root of unity.
Moreover, they extended the results of $[7]$.\par The main
aim of the present paper is to consider the spectrum preserving of certain products of operators on
the spaces of all self-adjoint
operators and also all bounded operators on a Hilbert space. \par We recall some notations. Let $\mathcal{H}$ be an
infinite dimensional complex Hilbert space, $\mathcal{B(H)}$ and
$\mathcal{S(H)}$ be the spaces of all bounded operators and
self-adjoint operators on $\mathcal{H}$, respectively. $I$ denotes
the identity operator and for any operator $A$ in
$\mathcal{B(H)}$, $\sigma (A)$ and $\vert A\vert $ denote the
spectrum of $A$ and the absolute value of $A$ that is equal to
$(A^*A)^{1/2}$, respectively. If $A\in \mathcal{B(H)}$, then
$f(t)=t^r$ is continuous and nonnegative on $\sigma (\vert A\vert
)$ for any positive rational number $r$. Hence $f$ belongs to
$\mathbf{C}(\sigma (\vert A\vert ))$. By the continuous functional
calculus, $f(A)=\vert A\vert ^r$ belongs to $\mathcal{B(H)}$. If
$P\in \mathcal{B(H)}$ is self-adjoint and $P^2=P$, then $P$ is
called projection. The set of all projections on $\mathcal{H}$ is
denoted by $\mathcal{P(H)}$. If $x,y\in \mathcal{H}$, then
$x\otimes y$ stands for the operator of rank at most one defined
by
$$(x\otimes y)z=<z,y>x  \hspace{.4cm}(z\in\mathcal{H}).$$
The set of all rank-one operators on $\mathcal{H}$ is denoted by
$\mathcal{F}_1(\mathcal{H})$. The set
of all rank-one projections on $\mathcal{H}$ is denoted by
$\mathcal{P}_1(\mathcal{H})$. \par Let $A\in \mathcal{B(H)}$ be a
finite-rank operator. On the set of all finite-rank operators on
$\mathcal{H}$, one can define the trace functional $\mathrm{tr}$ by
$$\mathrm{tr} A=\sum_{i=1}^{n} <x_i,y_i>,$$
where $A=\sum_{i=1}^{n} x_i\otimes y_i$. Then $\mathrm{tr}$ is a
well-defined linear functional.
\par Our main results are the follows.
\begin{theorem} Let $\mathcal{H}$ be a complex Hilbert space, $r$ and $s$ positive rational
numbers such that $r+s > 1$ and $\phi:\mathcal{S(H)}\rightarrow \mathcal{S(H)}$ a
surjective function which satisfies
$$\sigma(\vert A\vert ^rB\vert A\vert ^s) =
\sigma(\vert \phi(A)\vert ^r\phi(B)\vert \phi(A)\vert
^s)\hspace{.6cm}\leqno(*)$$ for all $A$ in
$\mathcal{P}_1(\mathcal{H})\cup \{I\}$ and $B$ in $\mathcal{S(H)}$. Then there exists a
bounded linear or conjugate linear bijection
$T:\mathcal{H}\rightarrow \mathcal{H}$ satisfying $T^*= T^{-1}$
such that
$$\phi(A) = TAT^*$$
for all $A\in \mathcal{S(H)}$.
\end{theorem}
\begin{theorem} Let $\mathcal{H}$ be a complex Hilbert space, $r$ a positive rational
number such that $r> 1$ and $\phi:\mathcal{S(H)}\rightarrow \mathcal{S(H)}$ a
surjective function which satisfies
$$\sigma(\vert A\vert ^rB)=
\sigma(\vert \phi(A)\vert ^r\phi(B))$$ for all $A$ in
$\mathcal{P}_1(\mathcal{H})\cup \{I\}$ and $B$ in $\mathcal{S(H)}$. Then there exists a
bounded linear or conjugate linear bijection
$T:\mathcal{H}\rightarrow \mathcal{H}$ satisfying $T^*= T^{-1}$
such that
$$\phi(A) = TAT^*$$
for all $A\in \mathcal{S(H)}$.
\end{theorem}
\begin{theorem} Let $\mathcal{H}$ be a complex Hilbert space, $r$ a positive rational
number such that $r> 1$ and $\phi:\mathcal{B(H)}\rightarrow \mathcal{B(H)}$ a
surjective function which satisfies
$$\sigma(\vert A\vert ^rB)=
\sigma(\vert \phi(A)\vert ^r\phi(B))$$ for all $A,B \in \mathcal{B(H)}$. Then there exists a
bounded linear or conjugate linear bijection
$T:\mathcal{H}\rightarrow \mathcal{H}$ satisfying $T^*= T^{-1}$
such that
$$\phi(A) = TAT^*$$
for all $A\in \mathcal{B(H)}$.
\end{theorem}
\begin{theorem} Let $r$ and $s$ be positive rational
numbers such that $r+s > 1$ and $\phi:\mathcal{B(H)}\rightarrow \mathcal{B(H)}$ a
surjective function which satisfies
$$\sigma(\vert A\vert ^rB\vert A\vert ^s) =
\sigma(\vert \phi(A)\vert ^r\phi(B)\vert \phi(A)\vert
^s)\hspace{.6cm}\leqno(*)$$ for all
$A,B \in \mathcal{B(H)}$. Then there exists a
bounded linear or conjugate linear bijection
$T:\mathcal{H}\rightarrow \mathcal{H}$ satisfying $T^*= T^{-1}$
such that
$$\phi(A) = TAT^*$$
for all $A\in \mathcal{B(H)}$.
\end{theorem}
\section{Proofs}
We use the following propositions and theorem to prove our results.\par
The following proposition is useful criteria for characterizing
rank-one projections.
\begin{proposition} Let $A\in \mathcal{P(H)}$. Then the following statements are
equivalent.\par
 $(a)$ $A$ is rank-one. \par
 $(b)$ $A\neq 0$ and $\mathrm{card} (\sigma(AT)\setminus \{0\})\leq 1$ for all $T\in \mathcal{S(H)}$.
\end{proposition}
\begin{proof} In order to complete the proof, it is sufficient to
prove that $(b)$ implies $(a)$. Assume on the contrary that there
exist $y_1,y_2$ in the range of $A$ such that are linearly
independent. Without loss of generality, suppose that $ \Vert y_1 \Vert =1$. If $$z_1= \frac{y_2-<y_2,y_1>y_1}{ \Vert y_2-<y_2,y_1>y_1 \Vert }$$ then $ \Vert z_1 \Vert =1$ and $<y_1,z_1>=0$.
Let $S=ay_1\otimes y_1+bz_1\otimes z_1$ for
some nonzero and different reals $a$ and $b$. It is easy
to check that $S\in \mathcal{S(H)}$, $ASy_1=ay_1$ and
$ASz_1=bz_1$, because $Ay_1=y_1$ and $Ay_2=y_2$. Hence $a,b\in \sigma(AS)$. This is a contradiction
and so the proof is completed.
\end{proof}
\begin{proposition} Let $A\in \mathcal{B(H)}$. Then the following statements hold.\par
 $(a)$ Let $A$ be a positive operator and $r$ a positive rational number. $A$ is rank-one if and only if $A^r$ is. \par
 $(b)$ $A$ is rank-one if and only if $ \vert A \vert $ is.
\end{proposition}
\begin{proof} $(a)$ If $A$ is rank-one, it is clear that $A^r$ is rank-one. Let
$m$ and $n$ be natural numbers such that $r=\frac{m}{n}$. It is
enough to prove that $A$ is rank-one when $A^m$ is rank-one,
because if $A^\frac{m}{n}$ is rank-one, then it is clear that
$A^m$ is rank-one. Assume that $A^m$ is rank-one but $A^{m-1}$
isn't. Then there are vectors $y_1,y_2\in \mathrm{Ran}A^{m-1}$ such that
$y_1$ and $y_2$ are linearly independent. Since $A^m$ is
rank-one, $Ay_1$ and $Ay_2$ are linearly dependent. So there
exist scalars $a,b\in \mathbb{C}$ with $ab\neq 0$ such that
$aAy_1+bAy_2=0$ which implies $ay_1+by_2\in \ker A$. Since $A$ is
positive, $ay_1+by_2\in (\mathrm{Ran}A)^\perp $. On the other hand,
$ay_1+by_2\in \mathrm{Ran}A$. So we obtain $ay_1+by_2=0$ which implies the
linear independence of $y_1$ and $y_2$ that is contradiction.
Therefore $A^{m-1}$ is rank-one. Inductively, it can be concluded
that $A$ is rank-one.
\par $(b)$ The statement "$A$ is rank-one if and only if $A^*A$
is" is proved similar to the previous part. Now since $\vert A\vert
=(A^*A)^{1/2}$, the assertion can be concluded from $(a)$.
\end{proof}
\begin{proposition} Let $A\in \mathcal{S(H)}$ and $x\in \mathcal{H}$. Then $Ax=0$ if and only if $ \vert A \vert x=0$.
\end{proposition}
\begin{proof} Assertion follows easily from the following equalities
$$<Ax,Ax>=<A^2x,x>=< \vert A \vert ^2x,x>=< \vert A \vert x, \vert A \vert x>.$$
\end{proof}
We use the following theorem to give the general forms of maps satisfying the hypothesis of the
mentioned main results.
\begin{theorem} $([14])$ Let $\mathcal{H}$ be a complex Hilbert space,
and let $L :\mathcal{S(H)}\rightarrow \mathcal{S(H)}$ be an
$ \Bbb{R}$-linear and weakly continuous operator. If
$L(\mathcal{P}_1(\mathcal{H})) = \mathcal{P}_1(\mathcal{H})$,
then there exists a bounded linear or conjugate linear bijection
$T:\mathcal{H}\rightarrow \mathcal{H}$ satisfying $T^*= T^{-1}$
such that $L(A) = TAT^*$ for all $A\in \mathcal{S(H)}$.
\end{theorem}
\par \vspace{.4cm}
\textbf{Proof of Theorem 1.1.}  We prove that $\phi$ is a
continuous additive bijection of $\mathcal{S(H)}$ that
preserves rank-one projections in both directions. In order to
prove this assertion we divide the proof into several steps. \par
\vspace{.4cm} Step 1. Let $B\in \mathcal{S(H)}$. $B=0$ if and
only if $\phi(B)=0$. \par Since $\phi $ is surjective, there
exists $B \in \mathcal{S(H)}$ such that $\phi(B)=0$. So by $(*)$
we have
$$\sigma(\vert A\vert ^{r}B\vert A\vert ^{s})=\{0\}$$
for all $A\in \mathcal{P}_1(\mathcal{H})$. This implies
$$\{0,<Bx,x>\}=\{0\}$$
for all $x$ from unit ball, which yields $B=0$.
\par \vspace{.4cm} Step 2. $\phi (I)=I$.
\par  First we assert that $\vert \phi (I)\vert ^{\frac{r+s}{2}} \phi(I) \vert \phi (I)\vert ^{\frac{r+s}{2}}=I$.
The condition $(*)$ yields
\begin{eqnarray*}
\sigma(\vert \phi (I)\vert ^{\frac{r+s}{2}} \phi(I) \vert \phi (I)\vert ^{\frac{r+s}{2}}) \cup \{0\}&=& \sigma(\vert \phi (I)\vert ^{r+s} \vert \phi (I) ) \cup \{0\}\\
&=&\sigma(\vert \phi (I)\vert ^r \phi(I) \vert \phi (I)\vert ^s ) \cup \{0\}\\
&=& \sigma (I) \cup \{0\}\\
&=& \{0,1\}.
\end{eqnarray*}
Hence $ \vert \phi (I)\vert ^{\frac{r+s}{2}} \phi(I) \vert \phi (I)\vert ^{\frac{r+s}{2}}$ is a projection. So in order to complete the proof of assertion, it is enough to show that $ \vert \phi (I)\vert ^{\frac{r+s}{2}} \phi(I) \vert \phi (I)\vert ^{\frac{r+s}{2}}$ is injective. First we show that $ \vert \phi (I)\vert ^{r+s}$ is injective. Because if $x \in \mathcal{H}$ such that $ \vert \phi (I)\vert ^{r+s}x=0$, then we have
$$< \vert \phi (I) \vert ^{r+s}x,x>=0$$ and hence
$$0= \sigma(\vert \phi (I)\vert ^{r+s} x \otimes x)= \sigma( \vert \phi (I)\vert ^r x \otimes x \vert \phi (I)\vert ^s ).$$ From $(*)$ we obtain $ \sigma (A)=0$ where $A$ is an operator such that $\phi (A
)=x \otimes x$. This implies that $A=0$. So by Step 1, $x=0$ and therefore $ \vert \phi (I)\vert ^{r+s}$ is injective. \par Now let $x \in \mathcal{H}$ such that $ \vert \phi (I)\vert ^{ \frac{r+s}{2}} \phi(I) \vert \phi (I)\vert ^{\frac{r+s}{2}}x=0$. The injectivity of $\vert \phi (I)\vert ^{r+s}$ yields the injectivity of $ \vert \phi (I)\vert ^{\frac{r+s}{2}}$. So $\phi(I) \vert \phi (I)\vert ^{\frac{r+s}{2}}x=0$. By Proposition 2.3, We have
$\vert \phi (I)\vert ^{\frac{r+s}{2}+1}x=0$. The condition $r+s > 1$ together with the injectivity of $ \vert \phi (I)\vert ^{r+s}$ and also $ \vert \phi (I)\vert ^{\frac{r+s}{2}}$ yields $x=0$ and this completes the proof of assertion. \par From assertion we can conclude that $\vert \phi (I)\vert ^{\frac{r+s}{2}}$ is invertible which multiplying by $\vert \phi (I)\vert ^{\frac{r+s}{2}}$ from right and then by $\vert \phi (I)\vert ^{- \frac{r+s}{2}}$ from left follows
 $$\phi(I) \vert \phi (I)\vert ^{r+s}=I$$ and hence $$\phi(I)=\phi(I) \vert \phi (I)\vert ^{r+s}\phi(I).$$
  So $\phi(I) \geq 0$, because we have
  $$\phi(I)=[\phi(I) \vert \phi (I)\vert ^{- \frac{r+s}{2}}][ \vert \phi (I)\vert ^{- \frac{r+s}{2}} \phi(I)]^*.$$
   $\phi(I) \geq 0$ and condition $(*)$ imply
$$1=\sigma(I)= \sigma(\vert \phi (I)\vert ^r \phi(I) \vert \phi (I)\vert ^s )
= \sigma(\phi (I)^{r+s+1}).$$
Therefore $\phi (I)=I$ and this completes the proof.
\par \vspace{.4cm} Step 3. $\phi$ is injective.
\par Let $\phi (B)=\phi
(B')$. By $(*)$ we obtain
$$\sigma(\vert A\vert ^{r}B\vert A\vert ^{s})=
\sigma(\vert A\vert ^{r}B'\vert A\vert ^{s})$$ for all $A\in
\mathcal{P}_1(\mathcal{H})$. This implies
$$\{0,<Bx,x>\}=\{0,<B'x,x>\}$$
and so
$$<(B-B')x,x>=0$$
for all $x$ from unit ball, which yields $B=B'$ and thus $\phi $
is injective.
\par \vspace{.4cm} Step 4. $\phi$ preserves the projections  and also the rank-one projections in both
directions.
\par  $(*)$ and Step 2 yield that $\sigma(B)=\sigma (\phi(B))$ for all $B\in
\mathcal{S(H)}$. So it is clear that $B$ is
a projection if and only if $\phi(B)$ is. Now let $P\in
\mathcal{P}_1(\mathcal{H})$. Since $\phi(P)$ is idempotent, by
$(*)$ we obtain
$$\sigma(PB) \setminus \{0\} = \sigma(\phi(P)\phi(B)) \setminus \{0\} \leqno(1)$$
for all $B \in \mathcal{S(H)}$. On the other hand, by Proposition 2.1 we have
$P\neq 0$ and $\mathrm{card} (\sigma(PB)\setminus \{0\})\leq 1$ for all $B\in \mathcal{S(H)}$.
This together with $(1)$, surjectivity of $\phi$ and Step 1 follows that $ \phi(P) \neq 0$ and $ \mathrm{card} (\sigma(\phi(P)B')\setminus \{0\})\leq 1$ for all $B'\in \mathcal{S(H)}$.
Again by Proposition 2.1 we infer that $\phi(P)$ is rank-one. Hence $\phi$ preserves the rank-one projections.
Since $\phi$ is injective and $\phi ^{-1}$ has the same properties of $\phi$, the rank-one projections are preserved by $\phi$ in both directions and this completes the proof.
\par \vspace{.4cm} Step 5. $\phi$ is additive.
\par Step 4 and $(*)$ imply that
$$ \mathrm{tr}(ABA)= \mathrm{tr}(\phi(A)\phi(B)\phi(A))$$
for all $A\in \mathcal{P}_1(\mathcal{H})$ and $B \in
\mathcal{S(H)}$. So the assertion can be proved in a very similar
way as the discussion in $[9]$.
\par \vspace{.4cm} Step 6. $\phi $ is continuous. \par We know that
$\sigma(A)=\sigma (\phi(A))$ for all $A\in \mathcal{S(H)}$. By
this fact and Step 5, we can conclude that $\phi $ is continuous.
\par \vspace{.4cm} So by above steps, $\phi$ is a continuous additive bijection
of $\mathcal{S(H)}$ preserving the rank-one projections in both directions. The forms
of such transformations is given in Theorem 2.4. So there exists a
bounded linear or conjugate linear bijection
$T:\mathcal{H}\rightarrow \mathcal{H}$ satisfying $T^*= T^{-1}$
such that
$$\phi(A) = TAT^*$$
for any $A\in \mathcal{S(H)}$. Therefore the proof of Theorem 1.1 is
complete. $\Box$
\par \vspace{.4cm}
\textbf{Proof of Theorem 1.2.} It is similar to the proof
of Theorem 1.1. $\Box$
\par \vspace{.4cm}
\textbf{Proof of Theorem 1.3.} We prove that $\phi$ is a
bijective linear map that preserves self-adjoint operators in
both directions. In order to prove this assertion we divide the
proof into several steps.\par \vspace{.4cm} Step 1. $\phi$ is
injective.
\par Let $\phi(B)=\phi(B')$. By $(**)$ we obtain
$$\sigma(\vert A\vert ^rB )=\sigma(\vert A\vert ^rB')$$
for all $A\in \mathcal{B(H)}$ and so for all $A\in
\mathcal{P}_1(\mathcal{H})$. This implies
$$\{0,<Bx,x>\}=\{0,<B'x,x>\}$$ and so
$$<(B-B')x,x>=0$$
for all $x$ from unit ball, which yields $B=B'$ and thus $\phi $
is injective.
\par \vspace{.4cm} Step 2. Let $A\in \mathcal{B(H)}$. $\vert A\vert ^r=0$ if and
only if $\vert \phi(A)\vert ^r=0$.
\par If $\vert A\vert ^r=0$, then by $(**)$ we have
$$\sigma(\vert \phi(A)\vert ^r \phi (B))=\{0\}$$
for all $B\in \mathcal{B(H)}$. This and surjectivity of $\phi$ follow
$$\{0,<\vert \phi(A)\vert ^rx,x>\}=\{0\}$$
for all $x\in \mathcal{H}$, which yields $\vert \phi(A)\vert ^r=0$. The converse is proved similarly.
\par \vspace{.4cm} Step 3. $\phi$ preserves rank-one operators in both
directions.
\par By Definition 2.2 in $[4]$ we have
\par $A$ is rank-one if and only if $A\neq 0$
and $\mathrm{card }(\sigma(AB) \setminus \{0\})\leq 1$ for all $B\in
\mathcal{B(H)}$.
\par Let $A$ be rank-one. By Proposition 2.2, $\vert A\vert ^r$ is rank-one and so $\vert A\vert ^r \neq 0$ and $\mathrm{card }(\sigma(\vert A\vert ^rB) \setminus \{0\})\leq 1$ for all $B\in
\mathcal{B(H)}$. By Step 2, $(**)$ and the surjectivity of $\phi$ we obtain $\vert \phi(A)\vert ^r \neq 0$ and $\mathrm{card }(\sigma(\vert \phi(A)\vert ^rB') \setminus \{0\})\leq 1$ for all $B'\in
\mathcal{B(H)}$. This implies that $\vert \phi(A)\vert ^r $ is rank-one. Again using Proposition 2.2 yields that $ \phi(A)$ is rank-one. The converse is proved similarly.
 \par \vspace{.4cm} Step 4. $\phi$ is linear.
\par Step 3 and $(**)$ yield
$$ \mathrm{tr} (\vert \phi(A)\vert ^r\phi (B))= \mathrm{tr} (\vert A\vert ^rB)$$
for all rank-one operator $A$ and all $B \in
\mathcal{B(H)}$. So the assertion can be proved in a very similar
way as the discussion in $[9]$.
\par \vspace{.4cm} Step 5. $\phi$ preserves self-adjoint operators in both directions.
\par Let $y$ be an arbitrary element from unit ball. By Step 3, there exists a rank-one
operator $A \in \mathcal{B(H)}$ such that $\phi(A)=y \otimes y $. Thus $\vert A\vert ^r=x \otimes x$ for some
$x \in \mathcal{H}$. Now let
$B$ be a self-adjoint operator. So by $(**)$ we obtain
$$\sigma(x \otimes B^*x)=\sigma(y \otimes \phi(B) ^*y)$$
which implies
$$<Bx,x>=<\phi(B)y,y>.$$
Since $<Bx,x>$ is real, $<\phi(B)y,y>$ is real. This together with the
arbitrariness of $y$ implies that $\phi(B)$ is self-adjoint. The converse is proved in a similar way.
\par \vspace{.4cm} So by above steps, $\phi$ is a bijective linear map such that preserves self-adjoint operators in
both directions and satisfies in $(**)$. The forms of such
transformations on $ \mathcal{S(H)}$ is given in Theorem 1.2. So
there exists a bounded linear or conjugate linear bijection
$T:\mathcal{H}\rightarrow \mathcal{H}$ satisfying $T^*= T^{-1}$
such that
$$\phi(A) = TAT^*$$
for all $A\in \mathcal{S(H)}$. Now let $A\in \mathcal{B(H)}$ be
arbitrary. Then there exist $A_1,A_2\in \mathcal{S(H)}$ such that
$A=A_1+iA_2$. So we have
$$\phi (A)=\phi (A_1+iA_2)=\phi (A_1)+i\phi(A_2)=TA_1T^*+iTA_2T^*=TAT^*.$$
The proof is complete. $\Box$
\par \vspace{.4cm}
\textbf{Proof of Theorem 1.4.} It is similar to the proof
of Theorem 1.3. $\Box$
\par \vspace{.4cm}{\bf Acknowledgements:} This research is partially
supported by the Research Center in Algebraic Hyperstructures and
Fuzzy Mathematics, University of Mazandaran, Babolsar, Iran.

\bibliographystyle{amsplain}

\begin{thebibliography}{10}\large

\bibitem{ta} {\sc B. Aupetit},
\textit{Spectrum-preserving linear mapping between Banach
algebras or Jordan–Banach algebras}, J. London Math. Soc. 62
(2000), 917-924.

\bibitem{ta} {\sc L. Baribeau, T. Ransford},
\textit{ Non-linear spectrum-preserving maps}, Bull. London Math.
Soc. 32 (2000), 8-14.

\bibitem{ta} {\sc G. Frobenius}, \textit{Über die Darstellung der endlichen Gruppen durch
lineare Substitutionen}, Sitzurngsber. Deutsch. Akad. Wiss.
Berlin (1897) 994-1015.

\bibitem{ta} {\sc R. Harte},
\textit{ On rank-one elements}, Studia Math. 117, (1995), 73-77.

\bibitem{co} {\sc J. Hou, C. K. Li, N. C. Wong}, \textit {Jordan isomorphisms and maps preserving spectra of
certain operator products }, Studia Math. 184 (2008), 31-47.

\bibitem{co} {\sc I. Kaplansky
}, \textit {Algebraic and Analytic Operator Algebras}, 1970

\bibitem{co} {\sc L. Huang and J. Hou}, \textit{Maps preserving spectral functions of operator products},
Chinese Ann. Math. Ser. A 28 (2007), 769–780.

\bibitem{ta} {\sc A. A. Jafarian, A. R. Sourour},
\textit{Spectrum-preserving linear maps}, J. Funct. Anal. 66
(1987) 255-261.

\bibitem{co} {\sc L. Moln\'{a}r}, \textit {Some characterizations of the automorphisms of $B(H)$ and $C(X)$},
Proc. Amer. Math. Soc. 130 (2002) 111-120.

\bibitem{co} {\sc M. Neal}, \textit {Spectrum preserving linear maps on $JBW^*$-triples},
Arch. Math. 79 (2002) 258-267.

\bibitem{co} {\sc T. Petek, P. \v{S}emrl}, \textit {Characterization of Jordan homomorphisms on $M_n$ using preserving properties}, Linear Algebra Appl. 269 (1998) 33-46.

\bibitem{ta} {\sc A.R. Sourour},
\textit{Invertibility preserving linear maps on $L(X)$}, Trans.
Amer. Math. Soc. 348 (1996) 13-30.

\bibitem{ta} {\sc A. Taghavi, R. Hosseinzadeh},
\textit{Linear maps preserving idempotent operators }, Bull. Korean Math. Soc. 47 (2010) 787-792.


\bibitem{ta} {\sc H. You, Sh. Liu, G. Zhang}, {Rank one preserving $R$-linear maps on spaces of
self-adjoint operators on complex Hilbert spaces}, Linear Algebra
Appl. 416 (2006) 568-579

\end{thebibliography}

\end{document}